\documentclass{amsart}
\usepackage{amsfonts,amssymb,mathrsfs,enumerate,enumitem,amsthm}
\usepackage{hyperref}
\usepackage{amsrefs,theoremref}
\theoremstyle{plain}
\newtheorem{theorem}[subsection]{Theorem}
\newtheorem{lemma}[subsection]{Lemma}
\newtheorem{prop}[subsection]{Proposition}

\theoremstyle{remark}

\newcommand{\abs}[1]{\left\lvert#1\right\rvert}
\newcommand{\bS}{\mathbb{S}}
\newcommand{\m}{\mathfrak{m}}

\newcommand{\norm}[1]{\left\lVert#1\right\rVert}
\DeclareMathOperator{\tr}{tr}

\begin{document}

        \author{Albert Chau$^1$}
        \address{Department of Mathematics,
                The University of British Columbia, Room 121, 1984 Mathematics
                Road, Vancouver, B.C., Canada V6T 1Z2} \email{chau@math.ubc.ca}

        \thanks{$^1$Research
                partially supported by NSERC grant no. \#327637-06}

        \author[L. Han]{Luke Kuo Han}
        \address{Department of Mathematics, The University of British Columbia,
1984 Mathematics Road, Vancouver, B.C.,  Canada V6T 1Z2}
\email{lukehan@math.ubc.ca}

\title{Bartnik Mass of surfaces under a Spectral non-negativity condition}
\maketitle

\pagestyle{plain}

\vspace{-3em}

\maketitle

\begin{abstract}
Let \(g\) be a smooth Riemannian metric and $H$ a positive function on
$\mathbb{S}^2$.  We prove that the Bartnik mass of the triple \((\bS^2,g,H)\)
is bounded above by $\sqrt{\abs{\mathbb{S}^2}_g/16\pi}$ provided the first
eigenvalue  $\lambda_1(g)$ of the operator $(-\Delta_g+K_g)$ is non-negative.
We prove a similar result assuming $H$ is only nonnegative and $\lambda_1(g)>0$.
These eigenvalue conditions, in particular, impose no lower bound on $K_g$
(even under an area constraint \cite{Mantoulidis_2015} \S3) and thereby extend
previous results which assume $K_g\geq 0$.

\end{abstract}

\section{Introduction}

Let \(g\) be a Riemannian metric and \(H\) be a function on \(\bS^2\).   A
classical problem in Riemannian general relativity and quasi-local mass theory
is the construction of admissible extensions of the Bartnik boundary data \(
(\bS^2,g,H)\).
Here an \begin{it}admissible extension\end{it} is a smooth asymptotically flat
3-manifold $(M, \gamma)$ such that
\begin{enumerate}[label=(\alph*)]
   \item \(M\) has non-negative scalar curvature,
   \item \(\partial M\) is diffeomorphic to $\bS^2$ and induces the metric $g$,
   \item \(\partial M\) has mean curvature $H$ (relative to the outward
normal),
   \item the interior of $M$ contains no closed minimal surfaces
\end{enumerate}
The ADM mass $\m_{ADM}(M,\gamma)$ of  each such admissible extension is well
defined \cite{ADM} and following  \cite{bartnik1989} the Bartnik mass of the
boundary data is defined to be

\begin{equation}\label{Bartnikmass}
\m_B(\bS^2,g,H)=\inf  \left\{\m_{ADM}(M,\gamma): \text{$(M,\gamma)$ an
admissible extension of $(\bS^2,g,H)$} \right\}
\end{equation}

There has been extensive work on constructing admissible extensions of given
Bartnik data \((\bS^2,g,H)\) under various conditions on $g$ and $H$ and estimating \(\m_B\).  We refer to \cite{Cabrera2017,
CabreraCederbaum2019,
chau-martens2021,
LS,
Mantoulidis_2015,
McCormick_2020,
McCormick_2024,
Miao_2009,
Miao_P_2024,
miao2019hawkingmassbartnikmass,
miao-xie,
shi-tam} and references therein.   We briefly describe some of the above cited works by
distinguishing two cases.

 The first case is when $H=0$ which corresponds to a minimal surface within
each admissible extension.   This case leads naturally to consider
\begin{equation}
\lambda_1(g):=  \text{ first eigenvalue of the operator $(-\Delta_g+K_g)$ on
$\bS^2$ }
\end{equation}
The condition $\lambda_1(g)\geq 0$ corresponds to the stability of the
minimal boundary sphere within a given admissible extension giving rise to a so
called  \textit{apparent horizon}  (see \cite{Mantoulidis_2015},
\cite{chau-martens2021} and references therein).
It was proved in \cite{Mantoulidis_2015} that
\begin{equation}\label{M-S bound}
    \m_B(\bS^2,g,H)\leq\sqrt{\abs{\bS^2}_g/16\pi}
\end{equation}
provided $H=0$ and $\lambda_1(g) > 0$, and this result was extended to include
the
degenerate case   $\lambda_1(g)= 0$ in  \cite{chau-martens2021}.   Here
$\abs{\bS^2}_g$ refers to the area of $\bS^2$ relative to $g$.  On the other
hand, the Riemann Penrose
inequality provides $\m_{ADM}(M,\gamma)  \geq \sqrt{\abs{\bS^2}_g/16\pi}$ for
every admissible extension and thus  equality actually holds in \eqref{M-S
bound} for these cases.   The proof in \cite{Mantoulidis_2015} introduced the
method of collar
constructions which consists of constructing a suitable metric $\gamma$ on
the collar $\bS^2\times[0, 1]$ and smoothly attaching, on one end, the given
Bartnik data and, on the other, an exterior Schwarzschild region with
controlled mass.   The collar construction method has since become a central
technique for estimating the Bartnik mass, and is used in the present work as
well as in many of the works cited above.

 The second case is when $H>0$ (see  \cite{Cabrera2017,
CabreraCederbaum2019,
chau-martens2021,
LS,
McCormick_2020,
McCormick_2024,
Miao_P_2024,
miao2019hawkingmassbartnikmass,
miao-xie,
shi-tam} and references therein).   In particular, it was shown in
\cite{Miao_P_2024}  that when $K_g\geq 0$ and $H$ is a positive function then
one has
\begin{equation}\label{M-SS bound}
    \m_B(\bS^2,g,H)\leq C(g, H) \sqrt{\abs{\bS^2}_g/16\pi}
\end{equation}
where $C(g, H)$ is a constant depending on $g, H$.  The same result was
established in \cite{chau-martens2021} in the case $H$ is constant, but there the
constant $C(g, H)$ was different and further satisfied $C(g, H)\leq 1$ (similar
bounds were also obtained with no assumption on $K_g$, but
only assuming $H$ is sufficiently large depending on $g$).  The assumption of
nonnegative Gauss curvature, $K_g\geq 0$, is common to the main results in all
of the works cited in this case.

 In this work, we extend the above results by relaxing the assumption $K_g\geq
0$ to assuming only that $\lambda_1(g)\geq 0$ while allowing $H$ to be any
positive function. Again, this is a strictly weaker assumption than \(K_g\geq
0\) (\cite{Mantoulidis_2015}\S3). We prove

\begin{theorem}\thlabel{mainthm}
    Let $g$ be a metric and \(H\) a positive function on
\(\bS^2\) such that $\lambda_1(g)\geq 0$.

 Then for any
\(m>\sqrt{\abs{\bS^2}_g/16\pi}\) there exists an admissible extension of
\((\bS^2,g,H)\) which is foliated by mean convex spheres and is  isometric to an exterior region of a mass-\(m\)
Schwarzschild manifold outside of a compact set.  In particular one has
\begin{equation*}
    \m_B(\bS^2,g,H)\leq \sqrt{\abs{\bS^2}_g/16\pi}
\end{equation*}
\end{theorem}

\par An exterior region of a mass-\(m\) Schwarzschild manifold here refers to
$$(\bS^2 \times (c,\infty),  r^2 g_* + \frac{1}{1 - \frac{2m}{r}}\, dr^2)$$
for $c>0$ where $g_*$ is the standard round metric on $\bS^2$.  The ADM mass of
this metric is equal to $m$ \cite{ADM} from which the last statement in Theorem
\ref{mainthm} immediately follows.

In Theorem \ref{secondthm} below we obtain a similar extension when \(H\)  is merely nonnegative and $\lambda_1(g)>0$, although here we must replacing condition (d) in the definition of admissibility with the condition (d)':  \(\partial M\) is outer-minimizing.     Corresponding to this replacement, we may define a Bartnik mass $\m_B' (\bS^2,g,H)$ exactly as in \eqref{Bartnikmass} and we refer to \cite{Jauregui2019} (Theorem 16) for a comparision between $\m_B(\bS^2,g,H)$ and $\m_B' (\bS^2,g,H)$ in particular when $\lambda_1(g)>0$.   We note that both conditions (d) and (d)' are satisfied when the extension is foliated by mean convex spheres as in Theorem \ref{mainthm} (d2) (see \cite{Cabrera2017}, Remark 1.1).

.

\begin{theorem}\thlabel{secondthm}
   Let $g$ be a metric and \(H\) a \textit{nonnegative} function on
\(\bS^2\) such that $\lambda_1(g)>0$.

 Then for any
\(m>\sqrt{\abs{\bS^2}_g/16\pi}\) there exists extension  of
\((\bS^2,g,H)\) which is admissible in the sense of condition (d)' and  which is isometric
to an exterior region of a mass-\(m\) Schwarzschild manifold outside of a
compact set.  In particular one has
\begin{equation*}
    \m'_B(\bS^2,g,H)\leq \sqrt{\abs{\bS^2}_g/16\pi}
\end{equation*}
\end{theorem}
In particular, if \(H\equiv 0\), \thref{secondthm} recovers
\cite{Mantoulidis_2015}.

We prove Theorem \ref{mainthm} and \thref{secondthm} following the collar
construction method from \cite{Mantoulidis_2015}.  In \S2 and \S3 we suitably
extend the Bartnik data to the collar \(\bS^2\times[0,1]\) under the respective
conditions of the theorems.  Our collar extension combines the ansatz used in
\cite{Mantoulidis_2015} together with the one used in \cite{chau-martens2021}
(see proof of Theorem 5.1).
Then in \S4 we invoke \cite[Proposition 2.1]{Cabrera2017} to extend our collar
metrics to an exterior Schwarzschild region with suitably controlled mass, and
thereby conclude the proof of \thref{mainthm} and \thref{secondthm}.

\section{Collar Extension of Theorem 1.1}

Fix Bartnik boundary data \((\bS^2,g,H)\) as in Theorem \ref{mainthm}.    We
will proceed to construct a suitable metric on \(\bS^2\times[0,1]\) extending
our boundary data.
Our construction will combine the collar ansatz used in \cite{Mantoulidis_2015}
together with the one used in \cite{chau-martens2021} (see proof of Theorem
5.1).
We recall the following lemma which is contained in \cite[Lemma
1.2]{Mantoulidis_2015} except the (v) \(\lambda_1(g)=0\) case appearing in
\cite[Lemma 2.1]{chau-martens2021}.

\begin{lemma}[\cite{Mantoulidis_2015} Lemma 1.2;   \cite{chau-martens2021}
Lemma 2.1]\thlabel{path lemma}
    There exists a smooth path of metrics $g(t)$ on $\bS^2$ for \(t\in[0,1]\)
such that
    \begin{enumerate}[label=\rm(\roman*)]
        \item \(g(0)=g\),
        \item \(g(1)\) is round.
        \item \(\dot g\equiv 0\) for \(t\in[1/2,1]\), and
        \item \(\frac{d}{dt}dA_g\equiv 0\) for \(t\in[0,1]\) where \(dA_g\) is
the area form of the metric \(g\).
\item If $\lambda_1(g)= 0$ then \(\lambda(t):=\lambda_1(g(t))\geq \alpha t\)
for some $\alpha>0$ and all $t\in [0, 1]$.
\\If \(\lambda_1(g)>0\), then \(\lambda(t)\geq\lambda\) for some \(\lambda>0\)
for all \(t\in[0,1]\).

\item There exists a smooth path of positive eigenfunctions
\(u(t,\cdot)>0\) on $\bS^2$ corresponding to \(\lambda(t)\).
    \end{enumerate}

\end{lemma}

Now we fix
\begin{itemize}
\item A path \(g(x, t)\) from Lemma \ref{path lemma}
\item A smooth  $\psi:[0,1]\to[0,1]$ : $\psi(t)=t$ for $t$ near $0$ and
$\psi(t)=0$ for $t$ near $1$.
\item A smooth  non-decreasing $\varphi:[0,\infty)\to[0,1]$ : $\varphi(t)=t$
for $t$ near $0$ and $\varphi(t)=1$ for $t\ge 1$.
\end{itemize}

For each $0<\varepsilon <(1/4)^{2/5}$ consider the smooth metric on
\(\bS^2\times[0,1]\) given by

\begin{equation}\label{gamma}
\gamma_{\varepsilon} =  \eta(x,t)\,g(x, t) + \Phi^2(t)u^2(x, t)\,dt^2
\end{equation}

where

\[\eta(x, t)=1+\varepsilon u(x, t) H(x) \psi(t)
+\varepsilon^{1/2} t^2\]

\[
\Phi(t)=
\begin{cases}
At + \varepsilon\,
& t\le \dfrac14, \\[1ex]
\varphi(A(t-1/4)) + \dfrac{A}{4} + \varepsilon\,
& \dfrac14 \le t \le \dfrac12, \\[1ex]
\dfrac{A}{4}+\varepsilon+1,
& \dfrac12 \le t.
\end{cases}\quad\quad\quad(A = \varepsilon^{-5/2})
\]

The metric \(\gamma_\varepsilon\) is in large the ansatz in
\cite{Mantoulidis_2015} with a bump linear term so that it produces the desired
boundary mean curvature while retaining the same gluing conditions at the end
of the collar.

\begin{lemma}\thlabel{collar lemma 1}
  There exists $0<\varepsilon_0<(1/4)^{2/5}$ such that for
all  $\varepsilon<\varepsilon_0$ the metric $\gamma_{\varepsilon}$ satisfies
the following on \(\bS^2\times[0,1]\)
 \begin{enumerate}[label=\rm(\roman*)]
        \item The mean curvature $H(t)$ of the sphere $\bS^2\times\{t\}$ is
positive for all $t\in [0, 1]$.  Moreover, $H(0)=H$ and $H(t)\to 0$ as
$\varepsilon\to 0$ and $t\to 1$.
  \item The scalar curvature $R_{\gamma_{\varepsilon}}$ is positive.
    \end{enumerate}
\end{lemma}

\begin{proof} For convenience in the proof we will write
$$\gamma_{\varepsilon}=  h(x, t)+ v(x, t)^2\,dt^2$$
 where  $h(x, t)= \eta(x,t)\,g(x, t)$ and $v(x, t)=\Phi(t)u(x, t)$.  We also
drop the subscript in $\gamma_{\varepsilon}$ and simply write $\gamma$ where
there is no confusion.   Finally, when describing dependencies of various
constants and quantities on the items above, we will ignore possible
dependencies on the functions $\psi$ and $ \varphi$.

\subsection*{Mean Curvature}
   The mean curvature $H(t)$ of the sphere $\bS^2\times\{t\}$ relative to
$\gamma$ is given by the following formula where $\nu=\frac{1}{v} \partial_t$
denotes the unit normal vector field on $\bS^2\times\{t\}$ in the positive $t$
direction.
  \[H(x, t)=\frac{1}{2}\tr_h(\mathcal{L}_\nu h) =\frac{\tr_h
h_t}{2v}=\frac{1}{2\eta}\frac{tr_g (\eta_t g+\eta g_t)}{\Phi
u}=\frac{1}{\eta}\frac{\varepsilon H(x)(u_t \psi + u \psi_t) +
2\varepsilon^{1/2} t }{\Phi u}\]
where we have used the fact that \(0=\frac{d}{dt}dA_g=\frac{1}{2}\tr_g g_t\)
by Jacobi's formula.  The second statement of part (i) of the Lemma follows
immediately using the definitions of $\eta$ and $\Phi$.   On the other hand,
recall $u(x, t)>0$ for all $(x, t)$  while $\psi(t)=t$ on $t\in [0, \delta]$
for some $\delta$ giving $$\varepsilon H(x)(u_t \psi + u \psi_t)
+2\varepsilon^{1/2} t > \varepsilon H(x) u_t  t+2\varepsilon^{1/2} t > 0$$
for $t\in [0, \delta]$ and $\varepsilon>0$ sufficiently small depending on $g,
H$.    Meanwhile we have
$$\varepsilon H(x)(u_t \psi + u \psi_t) +2\varepsilon^{1/2} t >0$$
 for $t\in [\delta, 1]$ and $\varepsilon>0$ sufficiently smaller still
depending on $g, H$.  The first statement in part (i) of the Lemma follows.

\subsection*{Scalar Curvature}
Recall that $g, H$ are as in Theorem \ref{mainthm}.  Thus  $\lambda_1(g)\geq 0$
and $H$ is positive.

  The scalar curvature of the warped product metric \(\gamma\) is given by the
following formula (see for example \cite{Mantoulidis_2015})
    \begin{equation}\label{scalar curvature}
\begin{split}
        R=&\left[2K_h-2v^{-1}\Delta_h v\right]+v^{-2}\left[-\tr_h
h_{tt}-\frac{1}{4}(\tr_h h_t)^2+\frac{v_t}{v}\tr_h
h_t+\frac{3}{4}\abs{h_t}_h^2\right]\\
=:& I +II
\end{split}
    \end{equation}
In the following, $C$ and $C_i$ (for each $i$) will denote a positive constants
depending
only on $g, H$ and may differ from line to line.

We may estimate

\begin{equation}\label{E1}
\begin{split}
  I&=\eta^{-1}(2K_g- \Delta_g \log \eta) -2v^{-1}\Phi\Delta_h u\\
&=2\eta^{-1}(K_g- u^{-1}\Delta_g u  -\frac{1}{2}\Delta_g \log \eta)\\
&\geq 2\eta^{-1} (\lambda(t) - C_0 \varepsilon t - C_0 \varepsilon^2 t^2 )   \\
&\geq C_1 t  \\
\end{split}
\end{equation}
where in the third line we have used that $\eta=(1+\varepsilon  u(x, t)H(x)
\psi(t) +\varepsilon^{1/2} t^2)$ and in the fourth line we have used that
$\eta$ is bounded above and below by positive constants depending only on $g,
H$, Lemma
\ref{path lemma} (v) and we have assumed $\varepsilon$ is suffciently small
depending on $g, H$.

\begin{equation}\label{E2}
\begin{split}
II=&v^{-2}\left[ (-\tr_h
h_{tt}-\frac{1}{4}(\tr_h h_t)^2+\frac{3}{4}\abs{h_t}_h^2)   +\frac{v_t}{v}\tr_h
h_t\right]\\
=&v^{-2}\left[ \left(-\tr_h h_{tt}-\frac{1}{4}(\tr_h
h_t)^2+\frac{3}{4}\abs{h_t}_h^2)   +\frac{u_t}{u}\tr_h h_t \right)
+\frac{\Phi_t}{\Phi}\tr_h h_t\right]\\
=&v^{-2}\left[ \left(-\tr_h h_{tt}-\frac{1}{4}(\tr_h
h_t)^2+\frac{3}{4}\abs{h_t}_h^2)   +\frac{u_t}{u}\tr_h h_t \right)
+\frac{\Phi_t}{\Phi}\frac{1}{\eta} \tr_g (\eta_t g + \eta g_t)\right]\\
=&v^{-2}\left[ \left(-\tr_h h_{tt}-\frac{1}{4}(\tr_h
h_t)^2+\frac{3}{4}\abs{h_t}_h^2)   +\frac{u_t}{u}\tr_h h_t \right)
+\frac{\Phi_t}{\Phi} \frac{2 \eta_t}{\eta} \right]\\
\geq& v^{-2}\left[-C_2+C_3 \frac{\Phi_t}{\Phi}\varepsilon\right]\\
 \end{split}
\end{equation}
Here we have estimated $\left(-\tr_h h_{tt}-\frac{1}{4}(\tr_h
h_t)^2+\frac{3}{4}\abs{h_t}_h^2+\frac{u_t}{u}\tr_h h_t\right)\geq-C_2$ using
$tr_h=\frac{1}{\eta}\tr_g$, the boundedness of $\eta$ noted above and that
derivatives (in space or time) of $h=\eta g$ are bounded (relative to $g$)
depending only on $g,H$.
We have also estimated $\eta_t=(\varepsilon H(x)(u_t\psi+ u \psi_t)
+2\varepsilon^{1/2} t) \geq C\varepsilon $ for some $C$. To see this we recall
that $\psi(t)=t$ for sufficiently small $t$ giving

$$\eta_t=(\varepsilon H(x)( t u_t + u ) +2\varepsilon^{1/2} t) \geq
C\varepsilon$$
for $t\in [0, \delta]$ for $\delta>0$ depending on $\psi, g, H$, using the
positivity of $u(x, t)$ and $H(x)$.  Meanwhile we also have

$$\eta_t=(\varepsilon H(x)(u_t \psi + u \psi_t) +2\varepsilon^{1/2} t) \geq
C\varepsilon$$
 for $t\in [\delta, 1]$ provided $\varepsilon$ is sufficiently small depending
on $g, H$.

Now we freeze the constants $C_1, C_2, C_3$ as they last appear above.

 Then for $t\in [0, \varepsilon^{3/2}]$ we can estimate as follows where in the
last inequality we have chosen $\varepsilon$ sufficiently small

\begin{equation}
\begin{split}
\frac{\Phi_t}{\Phi}(t)  \geq \frac{\Phi_t}{\Phi}(\varepsilon^{3/2})   &=
A/(A\varepsilon^{3/2}+ \varepsilon)\\
   &= 1 /(\varepsilon^{3/2}+ \varepsilon^{7/2})\\
&\geq C_2/(C_3 \varepsilon)\\
\end{split}
\end{equation}
and combining this with \eqref{E1}, \eqref{E2} gives $R\geq 0$  when $t\in [0,
\varepsilon^{3/2}]$.

For $t\in [\varepsilon^{3/2}, 1]$ we can continue the estimate in \eqref{E1} as
\begin{equation}\label{E9}
I \geq v^{-2}\left[ C_1 v^2  \varepsilon^{3/2} \right]  \geq v^{-2} \left[ C_1
(\varepsilon^{-5/2}\varepsilon^{3/2} +\varepsilon)^2 \varepsilon^{3/2}  \right]
\geq v^{-2}\left[ C_1 \varepsilon^{-1/2}\right]
\end{equation}
and combining this with \eqref{E2}, the fact that $\dot \Phi \geq 0$ and
choosing $\varepsilon$ sufficiently
smaller still if necessary gives $R\geq 0$ for $t\in [\varepsilon^{3/2}, 1]$.

This completes the proof of part (ii) of the Lemma.

\end{proof}

\section{Collar Extension of Theorem 1.2}
In this section, fix Bartnik boundary data \((\bS^2,g,H)\) as in
\thref{secondthm}. We construct a different collar metric which is still based
on the ansatz used in \cite{Mantoulidis_2015}.
\par For each \(0<\varepsilon\ll 1\), consider the following
\begin{itemize}
    \item A path \(g(x,t)\) from \thref{path lemma} with \(\lambda(t)\geq
\lambda>0\).
    \item A smooth map \(\phi:[0,1]\to[0,1]\) such that \(\phi(0)=0\),
\(\phi_t(0)=1\), \(\abs{\phi}\leq \varepsilon^2\), \(\abs{\phi_t}<3\),
\(\phi_{tt}<1\), and \(\phi(t)=0\) for \(t\in[1/2,1]\).\\The existence of such
\(\phi\) is provided in the Appendix.
\end{itemize}
Then we consider the smooth metric on \(\bS^2\times[0,1]\) given by
\[\rho_\varepsilon(x,t)=\eta(x,t)g(x,t)+\varepsilon^{-2}u^2(x,t)dt^2\]
where
\[\eta(x,t)=1+\varepsilon^{-1}u(x,0)H(x)\phi(t)+\varepsilon^{3/2} t^2\]

\begin{lemma}\thlabel{collar lemma 2}
  There exists $0<\varepsilon_0\ll 1$ such that for
all  $\varepsilon\leq\varepsilon_0$ the metric $\rho_{\varepsilon}$ satisfies
the following on \(\bS^2\times[0,1]\)
 \begin{enumerate}[label=\rm(\roman*)]
        \item The mean curvature \(H(t)\) satisfies \(H(0)=H\) and \(H(t)\to
0\) as \(\varepsilon\to 0\) and \(t\to 1\).
  \item The scalar curvature $R_{\rho_{\varepsilon}}$ is positive.
  \item The boundary \(\{t=0\}\) is outer-minimizing.
    \end{enumerate}
\end{lemma}
\begin{proof}
    For convenience we write \(\rho_\varepsilon=h(x,t)+v(x,t)^2dt^2\).
    \subsection*{Mean Curvature} By similar calculation from the last section,
    \begin{equation}\label{mean curvature 2}
        H=\frac{\tr_h h_t}{2v}=\frac{1}{2\eta}\frac{\tr_g(\eta_tg+\eta
g_t)}{\varepsilon^{-1}
u}=\frac{1}{\eta}\frac{\varepsilon^{-1}u_0H\dot\phi+2\varepsilon^{3/2}
t}{\varepsilon^{-1} u}
    \end{equation}
    which satisfies \(H(x,0)=\frac{H(x)\dot\phi(0)}{\eta(x,0)}=H(x)\). Since
\(\phi(t)=0\) for \(t\in[1/2,1]\), \(H(t)\to 0\) as \(t\to 1\) and
\(\varepsilon\to 0\). This proves (i).
    \subsection*{Scalar Curvature}
    We again have
    \begin{equation}
\begin{split}
        R=&\left[2K_h-2v^{-1}\Delta_h v\right]+v^{-2}\left[-\tr_h
h_{tt}-\frac{1}{4}(\tr_h h_t)^2+\frac{v_t}{v}\tr_h
h_t+\frac{3}{4}\abs{h_t}_h^2\right]\\
=:& I +II
\end{split}
    \end{equation}
Again, we use \(C_i\) to denote positive constants depending only on \(g\) and
\(H\) which may differ line to line.
\begin{align*}
    I&=\eta^{-1}(2K_g-\Delta_g\log\eta)-2v^{-1}\Delta_h u\\
    &=2\eta^{-1}(K_g-u^{-1}\Delta_g u-\frac{1}{2}\Delta_g\log\eta)\\
    &\geq 2\eta^{-1}(\lambda(t)-C_0\varepsilon^{-1}\abs{\phi})\\
    &\geq 2\eta^{-1}\lambda-C_0\varepsilon
\end{align*}
Here we used the fact that \(\eta\) is bounded only depending on \(g\) and
\(H\).
\par For term II, using the bounds on \(\phi\) and its derivatives, we have
\begin{align*}
    II&=v^{-2}\left[-\tr_h
h_{tt}-\frac{1}{4}(\tr_h h_t)^2+\frac{v_t}{v}\tr_h
h_t+\frac{3}{4}\abs{h_t}_h^2\right]\\
&\geq v^{-2}\left[-\tr_h
h_{tt}+\frac{v_t}{v}\tr_h h_t+\frac{1}{4}\abs{h_t}_h^2\right]\\
&\geq v^{-2}\left[-C_1\varepsilon^{-1}-C_2\right]\\
&\geq -C_1\varepsilon-C_2\varepsilon^2
\end{align*}
Then, it is clear that for sufficiently small \(\varepsilon\)
\[R=I+II=2\eta^{-1}\lambda-(C_0+C_1)\varepsilon-C_2\varepsilon^2>0\]
\subsection*{Outer-Minimizing}
We use a calibration argument. Let \(\pi:\bS^2\times[0,1]\to \bS^2\) be the
projection onto \(\{t=0\}\), and consider \(\omega=\pi^*dA_{g(0)}\).
    \par Given \((x,t)\in\bS^2\times[0,1]\), the tangent map
\(\pi_*:T_{(x,t)}(\bS^2\times[0,1])\to T_{x}\bS^2\) kills vectors in the
\(t\)-direction, so the supremum
    \[\norm{\omega}^*=\sup\{\omega(v,w):v,w\in
T_{(x,t)}(\bS^2\times[0,1]),\abs{v\wedge w}=1\}\]
    is attained by some orthonormal pair \(v,w\) tangent to
\(\{t=\text{constant}\}\). In particular,
\(\eta^{1/2}(x,t)\pi_*v,\eta^{1/2}(x,t)\pi_*w\) are \(g(t)\)-orthonormal, hence
    \[\omega(v,w)=dA_{g(0)}(\pi_*v,\pi_*w)=dA_{g(t)}(\pi_*v,\pi_*w)=\pm\eta^{-1}(x,t)\]
    Since \(\eta\geq 1\), this shows \(\norm{\omega}^*\leq 1\) so given any
surface \(\Sigma\) enclosing \(\{t=0\}\)
    \[\int_{\bS^2\times\{0\}}dA_{g(0)}=\int_{\bS^2\times\{0\}}\omega=\int_\Sigma\omega\leq\int_\Sigma\norm{\omega}^*dA_\Sigma\leq\int_\Sigma
dA_{\Sigma}\]
    where we used the closedness of \(\omega\), concluding the proof.
\end{proof}

\section{Gluing Exterior Schwarzschild Region and Proof of Theorems}
\subsection{Proof of \thref{mainthm}}
Fix Bartnik boundary data \((\bS^2,g,H)\) as in Theorem \ref{mainthm}.   Fix
some  metric $\gamma_{\varepsilon}$ on $\bS^2\times[0, 1]$  from Lemma
\ref{collar lemma 1}.   Recall the notation $$\gamma_{\varepsilon} =  h(x, t)+
v^2(x, t)\,dt^2:= \eta(x,t)\,g(x, t)+ \Phi(t)^2 u^2(x, t)\, dt^2$$ where $\eta,
\Phi$ are defined above the proof of Lemma \ref{collar lemma 1}.

  The Hawking mass of the triple $(\bS^2, h(1),
H(1))$ is defined as

\[
\m_H(\bS^2, h(1), H(1))
: =
\sqrt{\frac{|\bS^2|_{h(1)}}{16\pi}}
\left(
1-\frac{1}{16\pi}\int_{\bS^2} H(1)^2\, d\sigma
\right)=:I\cdot II
\]
It is clearly seen that $\eta(x,
1)\to 1$ as $\varepsilon \to 0$.   Combining
this with Lemma \ref{path lemma} (iv) gives $I\to
\sqrt{\frac{|\bS^2|_{g(0)}}{16\pi}}$ as $\varepsilon \to 0$, while combining
with Lemma \ref{collar lemma 1} (i) gives $II \to 1$ as $\varepsilon \to 0$.
In
summary, we have

\begin{equation}\label{BartnikvsHawking}
\m_H(\bS^2, h(1), H(1)) \to  \sqrt{\frac{|\bS^2|_{g(0)}}{16\pi}} \text{ as }
\varepsilon \to 0\end{equation}

\begin{prop}[{\cite[Proposition 2.1]{Cabrera2017}}]\thlabel{gluingprop}
Consider a metric
\[
\gamma = ds^2 + f(s)^2 g_*
\]
on $[a,b]\times \bS^2$, where $g_*$ is the standard metric on $\bS^2$ and
$f>0$ is a smooth function on $[a,b]$. Suppose that

\begin{enumerate}
\item $\gamma$ has positive scalar curvature;
\item $\Sigma_b := \{b\}\times \bS^2$ has positive mean curvature; and
\item $\m_H(\Sigma_b)\geq 0$.
\end{enumerate}

Then, for any $m_e > \m_H(\Sigma_b)$, there exists a smooth,
rotationally symmetric, asymptotically flat Riemannian $3$-manifold
$(M,\gamma)$ with boundary $\partial M$ and non-negative scalar curvature
such that

\begin{enumerate}
\item $M$, outside a compact set, is isometric to a spatial
Schwarzschild manifold of mass $m_e$;

\item $\partial M$ has a neighborhood $U$ that is isometric to
\[
(\left[a,\frac{a+b}{2}\right]\times S^2,\gamma);
\]

\item if $f' > 0$ on $[a,b]$, then $M$ can be constructed so that every
rotationally symmetric sphere in $M$ has positive constant mean
curvature.
\end{enumerate}
\end{prop}

 By the definition of $\Phi(t)$ and also Lemma \ref{path lemma}
(ii), (iii) we have that $\Phi(t)$ and $u(x, t)$ are both constant for $t\in
[1/2, 1]$.    Also by definition we have $\eta(x, t)$ is independent of $x$
for $t\in [1-\delta, 1]$ and some $0<\delta<1/2$.  Combining with Lemma
\ref{collar lemma 1} we conclude that
$\gamma$ satisfies the hypothesis of Proposition \ref{gluingprop} for the
interval $[a, b]=[1-\delta, 1]$,  and thus for arbitrary $m > \m_H(\bS^2, h(1),
H(1))$ we obtain an extension $(M, \gamma)$ as given in the Proposition.

 By Proposition \ref{gluingprop} and  Lemma \ref{collar lemma 1}, $(M, \gamma)$
is asymptotically flat with non-negative scalar curvature, while $\partial
M=\bS^2$ has induced metric $g$ and mean curvature $H$.  Moreover, the
foliating sphere $\bS^2\times \{t\}$ in $M$ has positive mean curvature for
each $t>0$ and the maximum principle implies that the interior of $(M, \gamma)$
contains no minimal surfaces.

Thus $(M, \gamma)$ is an admissible extension of the Bartnik data $(\bS^2, g,
H)$ with $\m_{ADM} (M, \gamma)=m$.  The statement in Theorem \ref{mainthm}
follows from noting \eqref{BartnikvsHawking} and that $\varepsilon>0$ and  $m >
\m_H(\bS^2, h(1), H(1))$  were both chosen arbitrarily.  This completes the
proof of Theorem \ref{mainthm}.

\subsection{Proof of \thref{secondthm}}
The exact same argument works for \(\rho_\varepsilon\). We only need to show
that \(\{t=0\}\) is outer-minimizing in the entire \(M\). By \eqref{mean
curvature 2}, the collar is mean-convex for some interval \([1-\delta,1]\).
Therefore, by \thref{gluingprop}, \(M\) is mean-convex beyond \(t=1-\delta\)
which means the end of the collar is outer-minimizing with respect to the
exterior. Thus \(\{t=0\}\) is outer-minimizing in the entire extension (see
\cite{McCormick_2020} Theorem 4.1).

\section{Appendix}
We construct the function \(\phi\) used in the proof of \thref{collar lemma 2}.

For \(a\geq 4\) consider
\[\varphi(t)=\begin{cases}
  -at^2+t=-a(t-\frac{1}{2a})^2+\frac{1}{4a}&t\in[0,1/2a]\\
  \frac{1}{4a}(\frac{1}{2a}-\frac{1}{4})^{-2}(t-\frac{1}{4})^2-\frac{1}{2a}(\frac{1}{2a}-\frac{1}{4})^{-3}(t-\frac{1}{4})^2(t-\frac{1}{2a})
& t\in[1/2a,1/4]\\
  0 & t\in[1/4,1]
\end{cases}\]
which satisfies \(\varphi\geq 0\), \(\varphi\in C^1\), \(\varphi(0)=0\),
\(\varphi'(0)=1\), \(\varphi\equiv 0\) on \([1/4,1]\),
\(\norm{\varphi}=\frac{1}{4a}\), \(\norm{\varphi'}=1\),
\(\sup\varphi''=\frac{3}{2a}(\frac{1}{4}-\frac{1}{2a})^{-2}\).

Let \(\chi\) be a smooth cutoff with \(\chi=1\) on \([1/4a,5/8]\), \(\chi=0\)
on \([0,1/8a]\cup[3/4,1]\). Consider
\(\phi_\epsilon=\chi(\varphi*\rho_\epsilon)+(1-\chi)\varphi\).
\par Since mollification preserves bounds: \(f\leq C\implies
f*\rho_\epsilon\leq C\), then
\begin{align*}
  \norm{\phi_\epsilon'}&=\norm{\chi'(\varphi*\rho_\epsilon)+\chi(\varphi'*\rho_\epsilon)-\chi'\varphi+(1-\chi)\varphi'}\\
  &\leq\norm{\chi'(\varphi*\rho_\epsilon-\varphi)}+\norm{\chi(\varphi'*\rho_\epsilon)}+\norm{(1-\chi)\varphi'}\\
  &\leq\norm{\chi'}\norm{\varphi*\rho_\epsilon-\varphi}+2\norm{\varphi'}\\
  &=\delta(\epsilon)+2\norm{\varphi'}
\end{align*}
\begin{align*}
  \phi''_\epsilon&\leq\norm{\chi''(\varphi*\rho_\epsilon-\varphi)}\norm{2\chi'(\varphi'*\rho_\epsilon-\varphi')}+\max[\chi(\varphi''*\rho_\epsilon)]+\max[(1-\chi)\varphi'']\\
  &\leq\delta(\epsilon)+2\max\varphi''
\end{align*}
where \(\delta(\epsilon)\to 0\) as \(\epsilon\to 0\) since \(\varphi'\) is
continuous. Then by taking \(a\) suitably large and \(\epsilon\) suitably
small, \(\phi_\epsilon\) is the desired choice of \(\phi\).


\end{document}